\documentclass[11pt]{amsart}
\usepackage{amsmath, amsthm, amssymb, euscript, enumerate}
\usepackage{pxfonts, graphicx, subfigure, epsfig}
\usepackage{mathtools, ulem, appendix, accents, color}
\usepackage{esvect}

\newcommand{\R}{{\mathbb R}}

\newcommand{\Ss}{{\mathbb S}}

\newcommand{\p}{\partial}

\newcommand{\ra}{\rightarrow}

\newcommand\norm[1]{\left\Arrowvert {#1}\right\Arrowvert}

\theoremstyle{plain}
\newtheorem{theorem}{Theorem}[section]

\newtheorem{corollary}[theorem]{Corollary}

\newtheorem{lemma}[theorem]{Lemma}

\theoremstyle{definition}

\theoremstyle{remark}

\numberwithin{equation}{section}

\title[Isolated singularity for semilinear elliptic equations]{Exact behavior around isolated singularity for semilinear elliptic equations with a log-type nonlinearity}

\author{Marius Ghergu}
\address{School of Mathematics and Statistics, University College Dublin, Belfield, Dublin 4, Ireland}
\email{marius.ghergu@ucd.ie}

\author{Sunghan Kim}
\address{Department of Mathematical Sciences, Seoul National University, Seoul 08826, Korea}
\email{sunghan290@snu.ac.kr}

\author{Henrik Shahgholian}
\address{Department of Mathematics, Royal Institute of Technology, 100 44 Stockholm, Sweden}
\email{henriksh@kth.se}


\begin{document}

\maketitle

\begin{abstract}
We study the semilinear elliptic equation 
\begin{equation*}
-\Delta u=u^\alpha |\log u|^\beta\quad\text{in }B_1\setminus\{0\},
\end{equation*}
where $B_1\subset\R^n$ with $n\geq 3$, $\frac{n}{n-2} < \alpha < \frac{n+2}{n-2}$ and $-\infty<\beta<\infty$. Our main result establishes that nonnegative solution $u\in C^2(B_1\setminus\{0\})$ of the above equation either has a removable singularity at the origin or behaves like 
\begin{equation*}
u(x) = A(1+o(1)) |x|^{-\frac{2}{\alpha-1}} \left(\log \frac{1}{|x|}\right)^{-\frac{\beta}{\alpha-1}}\quad\text{as } x\ra 0,
\end{equation*}
with 
\begin{equation*}
A=\left[\left(\frac{2}{\alpha-1}\right)^{1-\beta}\left(n-2-\frac{2}{\alpha-1}\right)\right]^{\frac{1}{\alpha-1}}.
\end{equation*}

\smallskip

\noindent
\end{abstract}


\section{Introduction}\label{section:intro}

Let $n\geq 3$ and $B_1$ be the unit open ball in $\R^n$. This paper is concerned with the behavior of nonnegative solutions of
\begin{equation}\label{eq:main}
-\Delta u = u^\alpha |\log u|^\beta\quad \text{in }B_1\setminus\{0\},
\end{equation}
where  $\alpha$ and $\beta$ are real numbers satisfying
\begin{equation}\label{eq:alpha-m}
\frac{n}{n-2} < \alpha < \frac{n+2}{n-2}\quad\text{and}\quad - \infty < \beta < \infty. 
\end{equation}
We say that $u$ is a nonnegative solution of \eqref{eq:main} if $u\in C^2(B_1\setminus\{0\})$ is nonnegative and satisfies \eqref{eq:main} pointwise. In addition, we say that a nonnegative solution $u$ of \eqref{eq:main} is {\it singular} if $u$ is unbounded in any punctured ball $B_r\setminus\{0\}$ with $0<r<1$. 
 
The case $\beta=0$ in \eqref{eq:main} is by now well understood; in their pioneering work \cite{GS1981}, Gidas and Spruck established a series of results that completely characterize the asymptotic behavior of local solutions of \eqref{eq:main} (with $\beta=0$). The main goal of this paper is to obtain similar results for \eqref{eq:main} when the exponents $\alpha$ and $\beta$ are in the range given by \eqref{eq:alpha-m}. 

Our main result is the following.
\begin{theorem}\label{th1}
Assume  $\alpha$ and $\beta$ satisfy \eqref{eq:alpha-m} and let $u$ be a nonnegative solution of \eqref{eq:main}. Then the following alternative holds:

\begin{enumerate}
\item[(i)] either $u$ has a removable singularity at the origin,

\item[(ii)] or $u$ is a singular solution and satisfies
\begin{equation}\label{main}
u(x)=(A+o(1)) |x|^{-\frac{2}{\alpha-1}}\left(\log \frac{1}{|x|}\right)^{-\frac{\beta}{\alpha-1}}\quad\mbox{as }x\to 0,
\end{equation}
where
\begin{equation}\label{const}
A=\left[\left(\frac{2}{\alpha-1}\right)^{1-\beta}\left(n-2-\frac{2}{\alpha-1}\right)\right]^{\frac{1}{\alpha-1}}.
\end{equation}
\end{enumerate}
\end{theorem}
For $\beta=0$ we recover the result in \cite[Theorem 1.3]{GS1981}. 
Let us note that in the case $\beta=0$, the approach in \cite{GS1981} relies to a large extend on the properties of the scaling function $u_\lambda(x)=\lambda^{\frac{2}{\alpha-1}}u(\lambda x)$ ($\lambda>0$). Thus, if $u$ is a solution of \eqref{eq:main} (with $\beta=0$) then, so is $u_\lambda$. 
A similar scaling is not available to us in case $\beta\neq 0$ due to the presence of the logaritmic term in \eqref{eq:main}. In turn, we shall take advantage of the result in \cite[Theorem 1.1]{CGS1989} which allows us to derive that singular solutions of \eqref{eq:main} are asymptotically radial. The exact asymptotic behavior \eqref{main} is further deduced by looking at the corresponding ODE of the scaled function $|x|^{\frac{2}{\alpha-1}}(\log \frac{1}{|x|})^{\frac{\beta}{\alpha-1}}u(x)$ in polar coordinates.

Asymptotic behavior of nonnegative singular solutions has been studied in various settings. In addition to the classical results \cite{GS1981} and \cite{CGS1989}, Korevaar et al. \cite{KMPS1999} derived the improved asymptotic behavior of the nonnegative singular solutions of $-\Delta u = u^{\frac{n+2}{n-2}}$ by a more geometric approach. Meanwhile, C. Li \cite{L1996} extended the result on the asymptotic radial symmetry of singular solutions of $-\Delta u = g(u)$ for a more general $g(u)$ considered in \cite{CGS1989}. Recently, the asymptotic radial symmetry has been achieved for other operators, such as conformally invariant fully nonlinear equations \cite{HLT2010, L2006}, fractional equations \cite{CJSX2014}, and fractional $p$-laplacian equations \cite{CL2017}. 

This paper extends the classical argument in \cite{GS1981} and \cite{CGS1989} to a log-type nonlinearity. One of the key observations is that from the asymptotic radial symmetry achieved in \cite{CGS1989} for nonnegative solutions of $-\Delta u = g(u)$ one can obtain an optimal asymptotic upper bound for $\frac{g(u)}{u}$. Hence, we are left with preserving the optimality by transforming $\frac{g(u)}{u}$ to $u$ under a suitable inverse mapping. 

This observation indeed allows us to consider a more general class of equations of type 
$$
-\Delta u = u^\alpha f(u)\quad\mbox{in }B_1\setminus\{0\},
$$ 
where $f$ is a slowly varying function at infinity under some additional assumptions. A typical example 
$$
f(u)=|\log^{(k_1)} u|^{\beta_1}|\log^{(k_2)} u|^{\beta_2}\cdots |\log^{(k_m)} u|^{\beta_m},
$$
where $k_i$'s are positive integer, $\beta_i$'s are real numbers, $\log^{(k)} u = \log (\log^{(k-1)}u)$ for $k\geq 2$ with $\log^{(1)}u= \log u$. However, we shall not specify the additional assumptions for the nonlinearity $f$ as they turn out to involve technical and cumbersome computations. Hence, we present the argument only with $f(u) = |\log u|^\beta$ in order to simplify the presentation. 

Throughout the paper, we shall write $f(x) = O(g(x))$ if $|f(x)|\leq C|g(x)|$ uniformly in $x$, where $C>0$ depends at most on $n$, $\alpha$ and $\beta$. We shall also use the notation $f(x) = o(g(x))$ as $x\ra 0$ to denote that $\frac{|f(x)|}{|g(x)|} \ra 0$ as $x\ra 0$.  


\section{Asymptotic Behavior around a Non-Removable Singularity}\label{section:nremv}

Let $\bar{u}(r)$  denote the spherical average of $u$ on the ball of radius $r$, that is,
\begin{equation}\label{eq:ub}
\bar{u}(r) = \fint_{\p B_r} u\,d\sigma. 
\end{equation}
The following result is a slight modification of Theorem 1.1 in \cite{CGS1989}.
\begin{theorem}\label{theorem:cgs} Let $u$ be a nonnegative solution of 
\begin{equation}\label{eq:cgs-pde}
-\Delta u = g(u) \quad\text{in } B_1\setminus\{0\},
\end{equation}
with an isolated singularity at the origin. Suppose that $g(t)$ is a locally Lipschitz function which in a neighborhood of infinity satisfies the conditions below: 
\begin{enumerate}[(i)]
\item $g(t)$ is nondecreasing in $t$;  
\item $t^{-\frac{n+2}{n-2}}g(t)$ is nonincreasing;
\item $g(t)\geq ct^p$ for some $p\geq\frac{n}{n-2}$ and $c>0$. 
\end{enumerate}
Then 
\begin{equation}\label{eq:cgs}
u(x) = (1+ O(|x|)) \bar{u}(|x|)\quad\text{as } x\ra 0.
\end{equation}
\end{theorem}
The original result in \cite[Theorem 1.1]{CGS1989} requires condition (i) above to be satisfied for all $t>0$ but a careful analysis of its proof shows that  this condition is enough to hold in a neighborhood of infinity.

It is not hard to see that $g(t) = t^\alpha |\log t|^\beta$ fulfills the conditions (i) -- (iii) in Theorem \ref{theorem:cgs} above. Moreover, it follows from \cite[Lemma 2.1]{CGS1989} that $u^\alpha|\log u|^\beta \in L^1(B_1)$ and $u$ is a distribution solution of \eqref{eq:main} in $B_1$: for any $\eta \in C_c^\infty(B_1)$, we have 
\begin{equation}\label{eq:main-dist}
-\int_{B_1} u\Delta \eta\,dx = \int_{B_1} u^\alpha|\log u|^\beta \eta\,dx.
\end{equation}

The next lemma provides an asymptotic upper bound for $\bar{u}$.

\begin{lemma}\label{lemma:upper} We have
\begin{equation}\label{eq:ub-upper}
\bar{u}(r) = O\left( r^{-\frac{2}{\alpha-1}} \left( \log \frac{1}{r} \right)^{-\frac{\beta}{\alpha-1}} \right),
\end{equation}
and
\begin{equation}\label{eq:ub'-upper}
\bar{u}'(r) = O\left( r^{-\frac{\alpha+1}{\alpha-1}} \left( \log \frac{1}{r} \right)^{-\frac{\beta}{\alpha-1}} \right),
\end{equation}
as $r \ra 0$. 
\end{lemma}

\begin{proof} Throughout this proof, $c>0$ depends at most on $n$, $\alpha$ and $\beta$, and may differ from one line to another. As mentioned earlier, we have $u^\alpha|\log u|^\beta \in L^1(B_1)$, and thus from the divergence theorem and \eqref{eq:main} we deduce that
\begin{equation}\label{eq:ub'}
-\bar{u}'(r) = \frac{c}{r^{n-1}}\int_{B_r} u^\alpha |\log u|^\beta\,dx.
\end{equation}
In particular, $\bar{u}(r)$ is monotone decreasing in $r$. Moreover, if \eqref{eq:ub-upper} holds, then one may easily derive \eqref{eq:ub'-upper} from \eqref{eq:ub'} and \eqref{eq:cgs}. 

Henceforth, we shall prove \eqref{eq:ub-upper}. Especially, we shall assume that $\bar{u}(r) \neq O(1)$ as $r\ra 0$, since the case $\bar{u}(r) = O(1)$ already satisfies \eqref{eq:ub-upper}. Under this assumption, we have $\bar{u}(r_k) \ra \infty$ for some $r_k\ra 0$. Then the monotonicity of $\bar{u}$ implies that $\bar{u}(r) \ra \infty$ as $r\ra 0$. 

Taking $r$ small enough, and using \eqref{eq:cgs} and the fact that $s\mapsto s^\alpha (\log s)^\beta$ is increasing for large $s$, we deduce that 
\begin{equation*}
-\bar{u}'(r) \geq cr \bar{u}^\alpha(r) (\log \bar{u}(r))^\beta. 
\end{equation*}
Hence, it follows from the assumption $\bar{u}(r) \ra \infty$ as $r\ra 0$ and the fact $\bar{u}'(r) < 0$ that  
\begin{equation*}
\int_{\bar{u}(r)}^\infty \frac{ds}{s^\alpha (\log s)^\beta} =  - \int_0^r \frac{\bar{u}'(r) dr}{\bar{u}^\alpha(r)(\log \bar{u}(r))^\beta} \geq  cr^2.
\end{equation*}
Note that for any sufficiently large $s$ satisfying $2|\beta|\leq (\alpha-1)\log s$, we have 
\begin{equation*}
-\frac{1}{\alpha-1}\frac{d}{ds} \left( \frac{1}{s^{\alpha-1}(\log s)^\beta}\right) = \left(1 - \frac{\beta}{(\alpha-1)\log s}\right)\frac{1}{s^\alpha(\log s)^\beta} \geq \frac{1}{2s^\alpha(\log s)^\beta},
\end{equation*}
whence we may proceed from the integral above as 
\begin{equation*}
\frac{1}{\bar{u}^{\alpha-1}(r) (\log\bar{u}(r))^\beta} \geq cr^2,
\end{equation*}
for sufficiently small $r>0$. Thus, we arrive at 
\begin{equation}\label{eq:upper2}
\bar{u}^{\alpha-1}(r)(\log \bar{u}(r))^\beta = O ( r^{-2} ) \quad\text{as } r\ra 0. 
\end{equation}

Setting $w(s)$ by the inverse function\footnote{$w$ is known as the Lambert $W$-function.} of $se^s$, we know that $\frac{s}{w(s)}$ is the inverse function of $s\log s$. Since $t^{\alpha-1}(\log t)^\beta = (c s\log s)^\beta$ with $s = t^{\frac{\alpha-1}{\beta}}$, we deduce from \eqref{eq:upper2} and the choice of $w$ that
\begin{equation}\label{eq:upper4}
\bar{u}(r) = O\left(  r^{-\frac{2}{\alpha-1}} w\left( r^{-\frac{2}{\beta}} \right)^{-\frac{\beta}{\alpha-1}} \right)\quad\text{as }r\ra 0.
\end{equation}
However, since $\log s - \log\log s \leq w(s) \leq \log s$, for sufficiently large $s$, we arrive at \eqref{eq:ub-upper}. 
\end{proof}

Let us next define
\begin{equation}\label{eq:psi}
\psi(t,\theta) = r^{\frac{2}{\alpha-1}} \left( \log \frac{1}{r} \right)^{\frac{\beta}{\alpha-1}} u(r,\theta),
\end{equation}
with $t=-\log r$ and $\theta\in\Ss^{n-1}$. 

\begin{lemma}\label{lemma:psi-pde} We have  
\begin{equation}\label{eq:psi-pde}
\psi_{tt} + \Delta_\theta \psi + a\psi_t - b\psi + \zeta^\beta\psi^\alpha = 0,
\end{equation}
for large $t>1$ and $\theta\in\Ss^{n-1}$, where 
\begin{align}
\label{eq:a}
a(t) &= \frac{4}{\alpha-1} - n + 2 - \frac{2\beta}{(\alpha-1)t},\\
\label{eq:b}
b(t) &= \left((n-2) - \frac{2}{\alpha-1} + \frac{\beta}{(\alpha-1)t}\right)\left(\frac{2}{\alpha-1} - \frac{\beta}{(\alpha-1)t}\right) - \frac{\beta}{(\alpha-1)t^2},
\end{align}
and 
\begin{equation}\label{eq:zeta}
\zeta (t,\theta) = \frac{2}{\alpha-1} - \frac{\beta}{\alpha-1}\frac{\log t}{t} + \frac{\log \psi(t,\theta)}{t}.
\end{equation}
\end{lemma}

\begin{proof} Take $r_0>0$ small enough such that $\log u > 0$ in $B_{r_0}$, and let us set $t_0 = - \log r_0$. In what follows, we take $t\geq t_0$ and $0<r\leq r_0$, unless stated otherwise. For the notational convenience, let us write 
\begin{equation*}
\phi(r) = r^{\frac{2}{\alpha-1}} \left( \log \frac{1}{r} \right)^{\frac{\beta}{\alpha-1}},
\end{equation*}
so that $\psi(t,\theta) = \phi(r)u(r,\theta)$. Since $\p_t = - r\p_r$ and $\p_{tt} = r\p_r + r^2 \p_{rr}$, we have  
\begin{equation}\label{eq:psi-pde-1}
\psi_{tt} + \Delta_\theta \psi = r^2 \phi \Delta u + (2r\phi' - (n-2)  \phi) r u_r  +  (r \phi' + r^2 \phi'') u,
\end{equation}
where the left and the right side are evaluated in $(t,\theta)$ and, respectively, in $(r,\theta)$, and by $\phi'$ and $\phi''$ we denoted $\frac{d\phi}{dr}$ and, respectively, $\frac{d^2\phi}{dr^2}$. Setting 
\begin{equation}\label{eq:eta}
\eta(r) = \frac{2}{\alpha-1} + \frac{\beta}{(\alpha-1)\log r},
\end{equation}
we observe that $r\phi' = \eta \phi$ and $r^2 \phi' = (\eta^2- \eta + r\eta')\phi$, and therefore,
\begin{equation}\label{eq:psi-pde-2}
\begin{split}
\psi_{tt} + \Delta_\theta \psi &= - r^2 \phi u^\alpha(\log u)^\beta + (2\eta - n + 2) r u_r \phi + (\eta^2 + r\eta') \phi u \\
& = - r^2 \phi u^\alpha (\log u)^\beta - (2\eta - n + 2)\psi_t + ( (n-2)\eta - \eta^2 + r\eta')\psi,
\end{split}
\end{equation}
where we used $\psi_t = -r\phi' u - r\phi u_r = -\eta\psi - r\phi u_r$ and $\psi = \phi u$ in deriving the second identity. 

In view of \eqref{eq:a} and \eqref{eq:b}, it is not hard to check that 
\begin{equation}\label{eq:a-re}
a(t) = 2\eta(r) - n + 2,
\end{equation}
and 
\begin{equation}\label{eq:b-re}
b(t) = (n-2)\eta(r) - \eta^2(r) + r\eta'(r).
\end{equation}
On the other hand, we know from \eqref{eq:psi} that 
\begin{equation*}
\log u(r,\theta) = \frac{2t}{\alpha-1} - \frac{\beta}{\alpha-1}\log t + \log \psi(t,\theta),
\end{equation*}
from which we may also deduce that 
\begin{equation}\label{eq:zeta-re}
\zeta(t,\theta) = \frac{\log u(r,\theta)}{\log\frac{1}{r}}.
\end{equation}
One may also notice from \eqref{eq:psi} that 
\begin{equation}\label{eq:psia}
r^2 \phi(r)u(r,\theta) = t^{-\beta} \psi^\alpha (t,\theta). 
\end{equation}
Hence, inserting \eqref{eq:a-re}, \eqref{eq:b-re}, \eqref{eq:zeta-re} and \eqref{eq:psia}, we arrive at the equation \eqref{eq:psi-pde}, which finishes the proof.
\end{proof} 

Let us define 
\begin{equation}\label{eq:psib}
\bar\psi (t) = \fint_{\Ss^{n-1}} \psi (t,\theta) d\theta  = r^{\frac{2}{\alpha-1}} \left( \log \frac{1}{r} \right)^{\frac{\beta}{\alpha-1}} \bar{u}(r),
\end{equation}
and 
\begin{equation}\label{eq:zetab}
\bar\zeta (t) = \frac{2}{\alpha-1} - \frac{\beta}{\alpha-1}\frac{\log t}{t} + \frac{\log \bar\psi(t)}{t}.
\end{equation}
Averaging \eqref{eq:psi-pde} over $\Ss^{n-1}$, we obtain 
\begin{equation}\label{eq:psib-ode}
\bar\psi'' + a\bar\psi' - b\bar\psi + \bar\zeta^\beta \bar\psi^\alpha + \fint_{\Ss^{n-1}} (\psi^\alpha\zeta^\beta - \bar\psi^\alpha\bar\zeta^\beta)d\theta =0,
\end{equation}
for large $t$.

\begin{lemma}\label{lemma:psi-psib} We have 
\begin{align}\label{eq:psi-psib}
&\psi(t,\theta) - \bar\psi(t) = \bar\psi(t) O(e^{-t}),\\
\label{eq:psi'-psib'}
&\left| \frac{\p}{\p t} (\psi(t,\theta) - \bar\psi(t)) \right| + \left| \nabla_\theta ( \psi(t,\theta) - \bar\psi(t,\theta) ) \right| = \bar\psi(t) O (e^{-t}),\\
\label{eq:psib-C01}
&\bar\psi(t)  = O(1)\quad\text{and}\quad \bar\psi'(t) = O(1),
\end{align}
as $t\ra\infty$. 
\end{lemma}

\begin{proof} In this proof, $C>0$ will depend on $n$ only and may differ from one line to another. The estimates in \eqref{eq:psib-C01} follow immediately from \eqref{eq:psib}, \eqref{eq:ub-upper} and \eqref{eq:ub'-upper}. Moreover, since $\psi(t,\theta) - \bar\psi(t) = r^{-\frac{2}{\alpha-1}} (\log\frac{1}{r})^{-\frac{\beta}{\alpha-1}} (u(r,\theta) - \bar{u}(r))$, \eqref{eq:psi-psib} can be easily deduced from \eqref{eq:cgs}. Thus, we are only left with proving  \eqref{eq:psi'-psib'}. 

For the notational convenience, let  $\bar{u}(x)$ be the $\bar{u}(|x|)$. Also let us denote by $A_r$ the annulus $B_{2r} \setminus \bar{B}_{\frac{r}{2}}$. From \eqref{eq:cgs}, we have 
\begin{equation*}
-\Delta (u - \bar{u}) = \bar{u}^\alpha |\log \bar{u}|^\beta O(r)\quad\text{in $A_r$ as $r\ra 0$}.
\end{equation*}
Therefore, it follows from the interior gradient estimates that 
\begin{equation*}
\begin{split}
|\nabla (u - \bar{u})| &\leq C\left( \frac{1}{r} \norm{u - \bar{u}}_{L^\infty(A_r)} + r^2 \norm{\bar{u}^\alpha |\log \bar{u}|^\beta}_{L^\infty(A_r)}  \right)\\
& \leq C\left( \norm{\bar{u}}_{L^\infty(A_r)} + r^2 \norm{\bar{u}^\alpha |\log \bar{u}|^\beta}_{L^\infty(A_r)} \right)\quad\text{on }\p B_r.
\end{split}
\end{equation*}

We regard \eqref{eq:main} as $-\Delta u=m(x) u$ in $B_1\setminus\{0\}$, where $m=u^{\alpha-1}|\log u|^\beta$. In view of \eqref{eq:upper2} and \eqref{eq:cgs} we have that $m(x)=O(|x|^{-2})$ so the Harnack inequality implies 
$$
\sup_{A_r} u \leq C\inf_{A_r} u.
$$ 
Using this observation along with \eqref{eq:cgs}, \eqref{eq:upper2} and the above gradient estimate, we find 
\begin{equation}\label{eq:u-ub}
|\nabla (u-\bar{u})| \leq C(\bar{u} + r^2 \bar{u}^\alpha |\log \bar{u}|^\beta ) \leq C\bar{u}\quad\text{on }\p B_r.
\end{equation}
Since $\psi(t,\theta) - \bar\psi(t) = r^{\frac{2}{\alpha-1}} (\log\frac{1}{r})^{\frac{\beta}{\alpha-1}} (u(r,\theta) - \bar{u}(r))$, \eqref{eq:psi'-psib'} follows from \eqref{eq:u-ub}, \eqref{eq:ub-upper} and \eqref{eq:cgs}. 
\end{proof}

\begin{corollary}\label{corollary:psi-psib} We have 
\begin{equation}\label{eq:psi-psib-int}
\fint_{\Ss^{n-1}} \left(\psi^\alpha (t,\theta) \zeta^\beta (t,\theta) - \bar\psi^\alpha(t) \bar\zeta^\beta (t) \right) d\theta = O(e^{-t})\quad\text{as }t\ra\infty.
\end{equation}
\end{corollary}

\begin{proof}
From \eqref{eq:zeta} and \eqref{eq:zetab} we know that
\begin{equation*}
\zeta(t,\theta) - \bar\zeta(t) = \frac{1}{t} \log \frac{\psi(t,\theta)}{\bar\psi(t)},
\end{equation*}
for any $t>0$ and $\theta\in\Ss^{n-1}$. Due to \eqref{eq:psi-psib}, we have 
\begin{equation*}
\zeta(t,\theta) - \bar\zeta(t) = \frac{1}{t}\log ( 1 + O (e^{-t})) = O\left(\frac{e^{-t}}{t}\right) \quad\text{as }t\ra\infty.
\end{equation*}
Using the above estimate together with Lemma \ref{lemma:psi-psib} we have
$$
\begin{aligned}
\psi^\alpha (t,\theta) \zeta^\beta (t,\theta) - \bar\psi^\alpha(t) \bar\zeta^\beta (t) &=(\psi^\alpha (t,\theta) -\bar\psi^\alpha(t))\zeta^\beta (t,\theta) +\bar\psi^\alpha (t)(\zeta^\beta (t,\theta)- \bar\zeta^\beta (t))\\
&=(\psi (t,\theta) -\bar\psi(t))O(1) +(\zeta(t,\theta)- \bar\zeta (t))O(1)\\
&=O(e^{-t})\quad\mbox{ as }t\to \infty.
\end{aligned}
$$
An integration over $\Ss^{n-1}$ in the above estimate will next lead us to \eqref{eq:psi-psib-int}.
\end{proof}

\begin{lemma}\label{lemma:psib-lim} We have either
\begin{equation}\label{eq:psib-lim-0}
\lim_{t\ra\infty} \bar\psi(t) = 0,
\end{equation}
or 
\begin{equation}\label{eq:psib-lim-A}
\lim_{t\ra\infty} \bar\psi(t) = A,
\end{equation}
with $A$ given by \eqref{const}. 
\end{lemma}

\begin{proof}
Let $\psi$ and $\bar\psi$ be defined by  \eqref{eq:psi} and \eqref{eq:psib}  respectively. Multiplying \eqref{eq:psib-ode} by $\bar\psi'$ and integrating it over $[t,T]$,  from \eqref{eq:psi-psib} -- \eqref{eq:psib-C01} and \eqref{eq:psi-psib-int} we find
\begin{equation}\label{eq:psib'-ode}
\frac{1}{2}\left[\bar\psi'^2\right]_t^T + \int_t^T a\bar\psi'^2 ds - \frac{1}{2}\int_t^T  b(\bar\psi^2)' ds + \frac{1}{\alpha+1}\int_t^T \bar\zeta^\beta (\bar\psi^{\alpha+1})'ds + O(e^{-t}) =0.
\end{equation}

By \eqref{eq:b} and  \eqref{eq:psib-C01} we have $b(t) = O(1)$ and $b'(t) = O(t^{-2})$, which leads us to 
\begin{equation}\label{eq:b-psib}
\int_t^T b(\bar\psi^2)' ds = \left[ b\bar\psi^2 \right]_t^T - \int_t^T b'\bar\psi^2 ds = O\left( 1 + \int_t^T \frac{ds}{s^2}\right) = O(1).
\end{equation}
Similarly, from \eqref{eq:zeta} and \eqref{eq:psib-C01} we find $\bar\zeta(t)=O(1)$ and $\bar\zeta'(t) = O(t^{-2}\log t)$, from which it follows that  
\begin{equation}\label{eq:zetab-psib}
\int_t^T \bar\zeta^\beta (\bar\psi^{\alpha+1})' ds = O\left( 1 + \int_t^T \frac{\log s}{s^2}ds\right) = O(1).
\end{equation}
Since $\alpha$ is chosen as in \eqref{eq:alpha-m}, we know from \eqref{eq:a} that $a(t)$ is positive and bounded away from zero for all large $t>1$. Thus, \eqref{eq:b-psib}, \eqref{eq:zetab-psib} and \eqref{eq:psib-ode} yield
\begin{equation}\label{eq:psib'-int}
\int_t^T \bar\psi'^2 ds = O(1).
\end{equation}
In view of \eqref{eq:psib-ode}, it follows from \eqref{eq:psib-C01} and \eqref{eq:psi-psib-int} that $\bar\psi''(t) = O(1)$ and hence $\bar\psi'^2(t)$ is uniformly Lipschitz for large $t>1$. Hence, we deduce that 
\begin{equation}\label{eq:psib'-lim}
\lim_{t\ra\infty} \bar\psi'(t) = 0.
\end{equation}

Now we multiply \eqref{eq:psib-ode} by $\bar\psi''$ and integrate it over $[t,T]$, which leads us to 
\begin{equation}\label{eq:psib''-ode}
\int_t^T (\bar\psi'')^2 ds + \frac{1}{2} \int_t^T a (\bar\psi'^2)' ds - \int_t^T b \bar\psi \bar\psi'' ds + \int_t^T \bar\zeta^\beta \bar\psi^\alpha \bar\psi'' ds + O(e^{-t}) = 0, 
\end{equation}
due to \eqref{eq:psi-psib} -- \eqref{eq:psib-C01} and \eqref{eq:psi-psib-int} as before. Note from \eqref{eq:a} that we have  $a(t)= O(1)$ and $a'(t) = O(t^{-2})$ as $t\to \infty$. Hence, from \eqref{eq:psib-C01} and \eqref{eq:psib'-int} we derive
\begin{equation}\label{eq:a-psib}
\int_t^T a(\bar\psi'^2)' ds = \left[a \bar\psi'^2\right]_t^T - \int_t^T a' \bar\psi'^2 ds = O(1). 
\end{equation}
On the other hand, since $\bar\psi\bar\psi'' = \frac{1}{2} (\bar\psi^2)'' - \bar\psi'^2$, a further integration by parts produces
\begin{equation}\label{eq:b-psib2}
\int_t^T b\bar \psi \bar \psi'' ds= \frac{1}{2}[b(\bar\psi^2)']_t^T-\int_t^T \left( \frac{1}{2} b' (\bar \psi^2)'+ b  \bar\psi'^2 \right) ds = O(1),
\end{equation}
where the second equality can be deduced analogously to the derivation of \eqref{eq:b-psib}. Similarly, we also observe that
\begin{equation}\label{eq:zeta-psib2}
\int_t^T \bar\zeta^\beta \bar\psi^\alpha \bar\psi'' ds = O(1). 
\end{equation} 
Due to \eqref{eq:a-psib}, \eqref{eq:b-psib2} and \eqref{eq:zeta-psib2}, \eqref{eq:psib''-ode} leads us to
\begin{equation}\label{eq:psib''-int}
\int_t^T (\bar\psi'')^2 ds = O(1).
\end{equation}
Differentiating \eqref{eq:psib-ode} with respect to $t$, we deduce from \eqref{eq:psib-C01}, \eqref{eq:psi-psib-int} and $\bar\psi''(t) = O(1)$ that $\bar\psi'''(t) = O(1)$. Therefore, $\bar\psi''^2(t)$ is uniformly Lipschitz for large $t>1$, from which combined with \eqref{eq:psib''-int} we obtain 
\begin{equation}\label{eq:psib''-lim}
\lim_{t\ra\infty} \bar\psi''(t) = 0.
\end{equation}

To this end, we shall pass to the limit in \eqref{eq:psib-ode} with $t\ra \infty$. Note that we have from \eqref{eq:b}
\begin{equation*}
\lim_{t\ra\infty} b(t) = \frac{2}{\alpha-1}\left(n-2-\frac{2}{\alpha-1}\right),
\end{equation*}
while it follows from \eqref{eq:zetab} and \eqref{eq:psib-C01} that 
\begin{equation*}
\lim_{t\ra\infty} \zeta(t) = \frac{2}{\alpha-1}.
\end{equation*} 
Although we do not know yet if $\bar\psi(t)$ converges as $t\ra\infty$, we still know from \eqref{eq:psib-C01} that it converges along a subsequence. Denoting by $\bar\psi_0$ a limit value of $\bar\psi(t)$ along a subsequence, say $t=t_j\ra\infty$, after passing to the limit in \eqref{eq:psib-ode} with $t=t_j$, we obtain from \eqref{eq:psi-psib-int}, \eqref{eq:psib'-lim} and \eqref{eq:psib''-lim} that
\begin{equation*}
\frac{2}{\alpha-1}\left(n-2-\frac{2}{\alpha-1}\right) \bar\psi_0 - \left(\frac{2}{\alpha-1}\right)^\beta \bar\psi_0^\alpha = 0. 
\end{equation*}
Thus, in view of \eqref{const}, we have 
\begin{equation}\label{eq:psib-lim}
\bar\psi_0=0\quad\mbox{or}\quad\bar\psi_0 = A. 
\end{equation}

Now the continuity of $\bar\psi$ implies that $\bar\psi(t)$ converges as $t\ra\infty$ (without extracting any subsequence) either to $0$ or $A$. If there are two distinct sequences $t_j\ra\infty$ and $t_j'\ra\infty$ such that $\bar\psi(t_j)\ra 0$ and $\bar\psi(t_j') \ra A$, then by the intermediate value theorem, there must exist some other $t_j''\ra\infty$ such that $\bar\psi(t_j'')\ra \frac{A}{2}$, which violates \eqref{eq:psib-lim}. Thus the proof is finished.  
\end{proof}

We are now in position to prove Theorem \ref{th1}.

\begin{proof}[Proof of Theorem \ref{th1}]
If \eqref{eq:psib-lim-A} is true, then in view of \eqref{eq:psib} we observe that
\begin{equation*}
\bar{u}(r)  = A(1+o(1)) r^{-\frac{2}{\alpha-1}} \left( \log \frac{1}{r} \right)^{-\frac{\beta}{\alpha-1}}\quad\text{as }r\ra 0.
\end{equation*}
Hence,  from \eqref{eq:cgs} we derive \eqref{main} and \eqref{const}, which establishes the proof for Theorem \ref{thm1} (ii).

Henceforth, let us suppose that 
\begin{equation}
\lim_{t\ra\infty} \bar\psi(t) = 0.
\end{equation}
The rest of the argument follows closely to the proof in \cite[Theorem 1.3]{CGS1989}. 

In view of \eqref{eq:a} and \eqref{eq:b}, we may rephrase \eqref{eq:psib-ode} as 
\begin{equation*}
\bar\psi'' + (a_0 + o(1))\bar\psi' - (b_0 + o(1))\bar\psi + \bar\zeta^\beta\bar\psi^\alpha + \fint_{\Ss^{n-1}} (\zeta^\beta \psi^\alpha - \bar\zeta^\beta \bar\psi^\alpha )d\theta = 0,
\end{equation*}
with 
\begin{equation*}
a_0 = \frac{4}{\alpha-1} - n + 2\quad\text{and}\quad b_0 = \frac{2}{\alpha-1} \left( n - 2 - \frac{2}{\alpha-1} \right).
\end{equation*}
Thus, the decay of $\bar\psi(t)$ is determined by the negative root of
\begin{equation*}
\lambda^2 + a_0 \lambda - b_0 = 0.
\end{equation*}
Since $a_0^2 + 4b_0 = (n-2)^2$, the root $\lambda$ is 
\begin{equation*}
\lambda = -\frac{1}{2} \left( a_0 + \sqrt{ a_0^2 + 4b_0 } \right) = - \frac{2}{\alpha-1}. 
\end{equation*}
Therefore, we have  
\begin{equation*}
\bar\psi (t) = O \left(e^{-\frac{2t}{\alpha-1}} \right)\quad\text{as }t\ra\infty.
\end{equation*}
In view of \eqref{eq:psib}, we obtain 
\begin{equation}\label{eq:ub-re}
\bar{u}(r) = \left(\log \frac{1}{r}\right)^{-\frac{\beta}{\alpha-1}} O(1)\quad\text{as }r\ra 0. 
\end{equation}

Now if $\beta > 0$, we deduce from \eqref{eq:ub-re} that $\bar{u}(r) \ra 0$ as $r\ra 0$, from which combined with \eqref{eq:cgs} it follows that $u(x) \ra 0$ as $x\ra 0$. Hence, the origin is a removable singularity. Similarly, if $\beta = 0$, \eqref{eq:ub-re} implies that $\bar{u}(r) = O(1)$, and thus, the origin is again a removable singularity.

Hence, we are only left with the case $\beta<0$. Since $u(x) = ( 1 + O(r))\bar{u}(r)$, \eqref{eq:ub-re} implies that 
\begin{equation*}
\int_{B_1} u^q \,dx \leq C \int_{B_1} \left(\log \frac{1}{|x|}\right)^{-\frac{\beta q}{\alpha-1}}\,dx \leq C \int_0^1 r^{n-1} \left(\log \frac{1}{r}\right)^{-\frac{\beta q}{\alpha-1}} \,dr \leq C
\end{equation*}
for each $q\geq 1$, for some const $C>0$ depending on $n$, $\alpha$, $\beta$ and $q$. Therefore, $u^\alpha |\log u|^\beta \in L^p (B_1)$ for any $p\geq 1$, and in particular for $p>n$. This implies that $\Delta u \in L^p (B_1)$ for $p>n$, so $u\in C^{1,\alpha}(B_{1/2})$ for $\alpha = 1- n/p$, proving again that the origin is a removable singularity. Thus the proof of Theorem \ref{th1} (i) is completed.
\end{proof}

\noindent{\it Acknowledgement.} This work was initiated in June 2017 when M. Ghergu was visiting the Royal Institute of Technology (KTH)  in Stocholm. The invitation and hospitality of the Department of Mathematics in KTH in greatly acknowledged.  S. Kim has been supported by National Research Foundation of Korea (NRF) grant funded by the Korean government (NRF-2014-Fostering Core Leaders of the Future Basic Science Program). H. Shahgholian has been supported in part by Swedish Research Council.

\end{document}